\documentclass[12pt]{amsart}
\usepackage{amsfonts}
\usepackage{amsmath}
\usepackage{amssymb}
\usepackage{amsthm}

\newcommand{\reals}{{\mathbb{R}}}
\newcommand{\eps}{{\varepsilon}}
\newcommand{\Lp}[2]{\|{#1}\|_{L^{#2}}}
\newcommand{\bfj}{{\mathbf{j}}}
\newcommand{\bfI}{{\mathbf{I}}}
\newcommand{\supp}{{\rm{supp}}}

\newcommand{\diam}{{\rm{diam}}}
\newcommand{\bfJ}{{\mathbf{J}}}
\newcommand{\scriptJ}{{\mathcal{J}}}

\newtheorem{theorem}{Theorem}
\newtheorem{lemma}{Lemma}
\newtheorem{prop}{Proposition}
\newtheorem*{cor}{Corollary}
\newtheorem{defn}{Definition}

\begin{document}


\title{$L^p$ improving multilinear Radon-like transforms}
\thanks{The author was supported in part by NSF grants DMS-040126 and DMS-0901569.}

\author{Betsy Stovall}
\address{Department of Mathematics, UCLA, Los Angeles, CA 90095-1555}
\email{betsy@math.ucla.edu}

\begin{abstract}
We characterize (up to endpoints) the $k$-tuples $(p_1,\ldots,p_k)$ for which certain $k$-linear generalized Radon transforms map $L^{p_1} \times \cdots \times L^{p_k}$ boundedly into $\reals$.  This generalizes a result of Tao and Wright.
\end{abstract}

\maketitle
\section{Introduction}

Let $n \geq 2$, $k \geq 2$, and for $1 \leq j \leq k$, let
$\pi_j:\reals^n \to \reals^{n-1}$ be a smooth submersion (i.e.\ $D\pi_j$ has
maximal rank at each point).  Without loss of generality, $\pi_j(0)
= 0$, $1 \leq j \leq k$.  We define an operator $S$, acting on
$k$-tuples of functions on $\reals^{n-1}$ by
\begin{align} \label{defS}
    S(f_1,\ldots,f_k):= \int_{\reals^n} \prod_{j=1}^k f_j \circ
    \pi_j(x) a(x) \,dx.
\end{align}
Here $a \in C_0^{\infty}(\reals^n)$ is a cutoff function with $a(0)
\neq 0$ whose support will be contained in a small neighborhood $V$
of 0.  We are interested in $k$-tuples $(p_1,\ldots,p_k)$ for which
$S$ satisfies
\begin{align} \label{stineq}
    |S(f_1,\ldots,f_k)| \leq C \prod_{j=1}^k
    {\Lp{f_j}{p_j}}_{(\reals^{n-1})},
\end{align}
where $C$ is a finite constant which depends on the $\pi_j$, the
$p_j$, and $a$, but not on the $f_j$.

This issue has been resolved (or resolved up to endpoints) in some special cases.  

First, when $k=2$, Tao and Wright in \cite{TW} characterized up to endpoints the pairs $(p_1,p_2)$ such that \eqref{stineq} is satisfied.  These bounds were reproved by Christ in \cite{ChRL} using partially alternative techniques.  More recently, in \cite{GressPolyEndpt}, Gressman used the techniques in \cite{TW} and \cite{ChRL} and some new details to establish restricted weak-type bounds at the endpoint when the fibres of the $\pi_j$ are polynomial curves.  

Second, when $k=n$ and the kernels of the differentials $d\pi_1,\ldots,d\pi_n$ span the tangent space to $\reals^n$ at every point, the bound \eqref{stineq} holds with $p_j=n-1$, $1 \leq j \leq n$.  When the $\pi_j$ are linear, this is the Loomis--Whitney inequality (\cite{LW}).  When the $\pi_j$ are not linear, the bound was proved by Bennett, Carbery, and Wright in \cite{BCW} (when the $\pi_j$ are in $C^3(\reals^n)$).  An alternative proof, using induction on scales, has recently been given in \cite{BHT} by Bejenaru, Herr, and Tataru (perturbed Loomis--Whitney in dimension 3) and as a special case of the result in \cite{BB} by Bennett and Bez.  The advantage of these two more recent results is that they give more quantitative bounds and treat a lower regularity case ($\pi_j \in C^{1,\beta}(\reals^n)$) than \cite{BCW}.

There is a significant gap between the special cases treated in \cite{TW} and \cite{BCW}.  In the first case, curvature of the fibres of the $\pi_j$ plays a role in determining the $L^p$ bounds, but the theory is restricted to bilinear operators.  In the second case, multilinear operators are allowed, but curvature plays no role.  Our goal in this article is to fill in this gap by characterizing (up to endpoints) the $k$-tuples $(p_1,\ldots,p_k)$ such that the bound \eqref{stineq} holds for all values of $k$ and smooth submersions $\pi_1,\ldots,\pi_k : \reals^n \to \reals^{n-1}$.  Our techniques are an adaptation of the method of refinements developed by Christ in \cite{CCC} and later applied in \cite{TW}, \cite{ChRL}, and \cite{BCW} as well as many other articles.  In particular, many of our arguments are adapted from \cite{ChRL} and \cite{TW}, though some new details are needed for the multilinear case.  To the extent possible, we have tried to follow the outline from those two works.

{\bf Notation.}  The notation we will employ is relatively standard, and in particular largely matches that in \cite{TW} and \cite{ChRL}.  As has become common in the harmonic analysis literature, we will hide constants in two types of symbols.  If $A$ and $B$ are two non-negative real numbers, then $A \lesssim B$ if there exists a (large) constant $C$ such that $A \leq CB$, and $A \ll B$ if there exists a (small) constant $c$ such that $A \leq cB$.  The chief difference between the symbols is in their use, $\lesssim$ typically appearing in the conclusion and $\ll$ in the hypothesis of a statement.  For instance, the statement, ``If $A \ll B$, then $f(A) \lesssim f(B)$,'' may be read as, ``There exist constants $c$ and $C$ such that whenever $A \leq cB$, we have $f(A) \leq Cf(B)$.''  Finally, $A \sim B$ if $A \lesssim B$ and $B \lesssim A$.  The dependence of the implicit constants will be specified as needed.

{\bf Acknowledgements.}  This work is adapted from part of the author's Ph.D.\ thesis, and she would like to thank her advisor, Mike Christ for suggesting this problem and for his help and advice on this project.  The author would also like to thank the anonymous referee for many helpful comments and suggestions.

\section{Restricted weak type inequalities and $L^p$-improving
operators}

We start by making some preliminary reductions in the $k$-tuples
$(p_1,\ldots,p_k)$ under consideration.  For the purposes of this
section, all implicit constants depend on $a$, the $\pi_j$, and the
$p_j$.

Suppose that $\sum_{j=1}^k p_j^{-1} \leq 1$.  Let $V$ be a bounded
set which contains $\supp (a)$.  Then by H\"{o}lder's inequality,
\[
    |S(f_1,\ldots,f_k)| \lesssim \prod_{j=1}^k {\Lp{f_j \circ \pi_j
    \cdot \chi_V}{p_j}}_{(\reals^n)} \lesssim \prod_{j=1}^k
    {\Lp{f_j}{p_j}}_{(\reals^{n-1})},
\]
where the last inequality follows from our assumption that $\pi_j$
is a submersion and by the boundedness of $V$.

Now suppose that $p_i < 1$ for some $1 \leq i \leq k$.  For $0 <
\delta \ll 1$, let $f_i = f_{\delta}$ be the characteristic function
of the ball of radius $\delta$ centered at 0 in $\reals^{n-1}$.  For
$j \neq i$, let $f_j$ be the characteristic function of the ball of
radius $1$ centered at 0 in $\reals^{n-1}$.  Then since $a(0) \neq
0$, and since $\pi_j$ is a submersion,
\[
    |S(f_1,\ldots,f_k)| \sim \delta^{n-1}
\]
and
\[
    \prod_{j=1}^k \Lp{f_j}{p_j} \sim \delta^{(n-1)/p_i}.
\]
Letting $\delta \to 0$, \eqref{stineq} cannot hold.

Henceforth, we will consider only those $k$-tuples
$p=(p_1,\ldots,p_k)$ satisfying
\begin{align} \label{p>1}
    \sum_{j=1}^k p_j^{-1}>1 \hspace{.5cm} \text{and} \hspace{.5cm} 1
    \leq p_j \leq \infty \hspace{.5cm} \text{for $1 \leq j \leq k$}.
\end{align}
We say that $S$ is $L^p$-improving if it satisfies \eqref{stineq}
for some $k$-tuple $p$ satisfying \eqref{p>1}.  This terminology is
motivated by the case when $k=2$ and the fibers of $\pi_1$ and
$\pi_2$ are transverse near 0.  In this case, \eqref{stineq} is equivalent to
boundedness of the operator $T$ defined by
\[
    Tf(x) = \int_{\pi_2^{-1}\{x\}}f(y)a(y)d\sigma_x(y)
\]
from $L^{p_1}(\reals^n)$ to $L^{p_2'}(\reals^n)$. (Here $d\sigma_x$ is equal to arclength measure on $\pi_2^{-1}\{x\}$.)  Additionally, \eqref{p>1} is equivalent to $p_2' > p_1$, and since
$L^{p_2'}(V) \subset L^{p_1}(V)$ for $V$ bounded, we think of $Tf$
as lying in a better space than $f$.

We will focus on establishing restricted weak type inequalities, i.e. in proving \eqref{stineq} in the case when each $f_j$ is the characteristic function of a Borel set.  We note that in order to prove that $S$ is of restricted weak type $(p_1,\ldots,{p_k})$, it is enough to show that whenever $\Omega \subset \reals^n$ is a Borel set, we have
\begin{align}\label{wkineq2}
\int \chi_{\Omega}(x)a(x)dx \lesssim |\pi_1(\Omega)|^{1/p_1} \cdots
|\pi_k(\Omega)|^{1/p_k},
\end{align}
with the implicit constant independent of $\Omega$.  This can be reformulated as a lower bound on $\Omega$ as follows.  We assume that $\Omega \subset \supp(a)$, and for $1 \leq j \leq k$, we define
\begin{align} \label{defalpha}
    \alpha_j := \frac{|\Omega|}{|\pi_j(\Omega)|}.
\end{align}
We observe that by the coarea formula, $\alpha_j$ is approximately equal to the average size (in terms of euclidean arclength) of the intersection of the fibres of $\pi_j$ with $\Omega$.  Additionally, under the assumption \eqref{p>1}, \eqref{wkineq2} is equivalent to the inequality 
\begin{align} \label{wkineq3}
    \alpha^b:= \alpha_1^{b_1} \cdots \alpha_k^{b_k} \lesssim
    |\Omega|,
\end{align}
where $b = b(p) = (b_1,\ldots,b_k)$ is defined by
\begin{align} \label{defb}
    b_i = \frac{p_i^{-1}}{\sum_{j=1}^k p_j^{-1}-1}, \quad 1 \leq i \leq k.
\end{align}
It is this inequality that we will try to prove.

\section{Vector fields and statement of results}
Associated to each of the submersions is a (nonunique) $C^{\infty}$
vector field $X_j$ on $\reals^n$ which is nonvanishing and tangent
to the fibers of $\pi_j$.

In \cite{CNSW}, Christ, Nagel, Stein, and Wainger proved that if
$k=2$, then $S$ is $L^p$-improving if and only if the $X_j$ and
their iterated Lie brackets span the tangent space to $\reals^n$ at
each point where $a \neq 0$.  In Section~\ref{cnswmlpf}, we will prove
the necessity portion of this theorem for $k > 2$:
\begin{theorem} \label{cnswml}
    If $S$ is $L^p$-improving, then the $X_j$ and their iterated Lie
    brackets span the tangent space to $\reals^n$ at each point
    where $a \neq 0$.
\end{theorem}

The spanning of the tangent space by the iterated Lie brackets is known as the H\"{o}rmander condition.

The results of \cite{CNSW} characterized $L^p$-improving operators
but did not show for which pairs $(p_1,p_2)$ the inequality
\eqref{stineq} holds.  In \cite{TW} Tao and Wright determined, up to
endpoints, which pairs $(p_1,p_2)$ as in \eqref{p>1} satisfy
\eqref{stineq}.  A later proof of this theorem, along the lines of \cite{TW} but with some simplifications, is due to Christ in
\cite{ChRL}.  The bulk of this article will be devoted to showing
that the result of Tao and Wright extends to the case $k \geq 2$.

To state the result, we review a few definitions from \cite{TW}, generalized to the multilinear setting.

A {\it word} is a $d$-tuple $w \in \{1,\ldots,k\}^d$ for some $d \geq 1$, and $W$ denotes the set of all words.  If $w \in W$, its degree is the $k$-tuple whose $j$-th component is the number of entries of $w$ which equal $j$.  Finally, to each $w \in W$ is associated a vector field $X_w$, defined by the recursive equation
\[
    X_{(w,j)}:= [X_w,X_j].
\]
For example, if $k \geq 3$,
\[
    X_{(1,2)} = [X_1,X_2] \qquad X_{(1,2,3)} = [[X_1,X_2],X_3].
\]

Crucially, by antisymmetry and the Jacobi identity, each iterated Lie bracket
of the $X_j$ may be written as a linear combination (with constant coefficients) of vector fields
$X_w$ (cf.\ \cite[Lemma~4.4]{H}).  For instance,
\[
    [[X_1,X_2],[X_3,X_4]] = -X_{(1,2,4,3)}+X_{(1,2,3,4)}.
\]
Therefore the H\"{o}rmander condition is equivalent to
the statement that the vector fields $X_w$, with $w \in W$, span the
tangent space to $\reals^n$ at each point where $a \neq 0$.

Given an $n$-tuple $I=(w_1,\ldots,w_n)$ of words, we denote by $\lambda_I$ the determinant
\[
\lambda_I(x) := \det(X_{w_1}(x),\ldots,X_{w_n}(x)),
\]
and by $\deg(I)$ the $k$-tuple 
\[
\deg(I) := \deg(w_1) + \cdots + \deg(w_n).
\]
Finally, we recall one more definition before stating the main theorems.  The Newton polytope $P$ of the vector fields
$X_1,\ldots,X_k$ is the closed convex hull of the set of points
\[
    \{(b_1,\ldots,b_k) \geq \deg(I): \text{$I \in W^n$ and
    $\lambda_I(0) \neq 0$}\}.
\]
Here $(b'_1,\ldots,b'_k) \geq (b_1,\ldots,b_k)$ if $b'_i \geq b_i$ for
each $1 \leq i \leq k$.  Hence if $P \neq \emptyset$, then $P$ has
nonempty interior.  Recall the definition \eqref{defb} of $b(p)$ for
$p=(p_1,\ldots,p_k)$ satisfying \eqref{p>1}.  The next theorem
sharpens Theorem~\ref{cnswml}.

\begin{theorem} \label{nec}
    If $p \in [1,\infty]^k$ satisfies \eqref{p>1}, and $b(p)$ does
    not lie in $P$, then $S$ is not even of restricted weak-type
    $(p_1,\ldots,p_k)$.
\end{theorem}

The above theorem and the next one almost completely characterize
those $k$-tuples $p$ for which $S$ satisfies \eqref{stineq}.

\begin{theorem} \label{suff}
    If $p \in [1,\infty]^k$ satisfies \eqref{p>1} and $b(p)$ lies in
    the interior of $P$, then, provided $\supp(a)$ is contained in a
    sufficiently small neighborhood of 0, $S$ satisfies
    \eqref{stineq}, i.e. $S$ is of strong-type $(p_1,\ldots,p_k)$.
\end{theorem}

We have as a corollary to Theorems~\ref{cnswml}, \ref{nec}, and~\ref{suff} the following.

\begin{cor}
    The operator $S$ is $L^p$-improving if and only if the iterated
    Lie brackets of the $X_j$ span the tangent space to $\reals^n$
    at each point where $a \neq 0$.
\end{cor}

We will give a more geometric formulation of Theorem~\ref{suff} later; see Theorem~\ref{thm:geometric}.

\section{Related work}

There is an extensive bibliography in \cite{TW} to which we direct the interested reader.  We will focus here on some more recent results.

\subsection*{Endpoint bounds.}  The main theorems here and in \cite{TW} and \cite{ChRL} do not establish boundedness of the operator $T$ at the Lebesgue endpoints (the points $p=(p_1,\ldots,p_k)$ with $b(p)$ equal to a vertex of the Newton polytope).  At least in the case of real analytic $\pi_j$, it is likely that the proof and not the operator is at fault for this omission.  See remarks in \cite[Section 11]{ChRL} for some conjectures and remarks related to endpoint bounds in the case $k=2$.

In a few special cases of the Tao--Wright theorem, strong-type endpoint bounds are known.  When $k=2$, one example of our operator is $S(f_1,f_2) := \langle Tf_1, f_2 \rangle$, where 
\[
Tf(x) := \int f(x-\gamma(t))\, a(x,t)\, dt.
\]
Here $\gamma:\reals \to \reals^n$ is a parametrized curve.  When $\gamma(t) = (t,\ldots,t^n)$, the author has proven endpoint bounds for $T$ in high dimensions in \cite{BSCCC}, sharpening the restricted weak-type result due to Christ in \cite{CCC} and extending lower dimensional results of Littman in \cite{Litt}, and Oberlin in \cite{Obe1}, \cite{Obe2}, \cite{Obe3}.  In another recent article \cite{DLW}, Dendrinos, Laghi, and Wright have established strong-type endpoint bounds for convolution with affine arclength measure along polynomial curves $\gamma$ in dimension 3, extending work of Oberlin in \cite{Obe4}, and this result was generalized to higher dimensions by the author in \cite{BSpoly_endpt}.  The work \cite{Laghi} of Laghi is in a similar vein.

Finally, in \cite{GressPolyEndpt} Gressman has settled the question of restricted weak-type boundedness at the endpoint in the polynomial case of the Tao--Wright theorem.  The arguments in \cite{GressPolyEndpt} extend to the multilinear setting with no difficulty.

\subsection*{Another multilinear operator.}  In \cite{BCW}, Bennett, Carbery and Wright have proved the following non-linear generalization of the Loomis--Whitney inequality:  If $k=n$ and the vector spaces $X_1,\ldots,X_n$ span the tangent space to $\reals^n$ at 0, then for $a$ having sufficiently small support (containing 0), $|S(f_1,\ldots,f_n)| \lesssim \prod_{j=1}^n \|f_j\|_{L^{n-1}(\reals^{n-1})}$.    The articles \cite{BB} and \cite{BHT} contain proofs of this generalization by induction on scales and establish the generalization of the Loomis--Whitney inequality under much lower regularity assumptions (though \cite{BB} is more general, addressing for instance certain cases wherein the fibres have unequal dimensions).  Our result gives a partial generalization of this bound (in the perturbed Loomis--Whitney case).  On the one hand, our methods cannot be used to obtain the endpoint $(n-1,\ldots,n-1)$ as this corresponds to $b = (1,\ldots,1)$.  On the other hand, even in the case $k=n$, our theorem establishes new bounds for certain $\pi_1,\ldots,\pi_n$.

\subsection*{Multi-parameter Carnot--Carath\'{e}odory Balls}  Independently of this work, in \cite{Street} Street has generalized the work of Nagel--Stein--Wainger and Proposition 4.1 of \cite{TW} for multi-parameter Carnot--Carath\'{e}odory balls, as well as analyzing the situation when the iterated Lie brackets span a proper subspace of $\reals^n$.  This work does not however address certain issues which we use to bound the multilinear operator.

\section{Proof of Theorem~\ref{cnswml}} \label{cnswmlpf}

In \cite{CNSW}, Christ, Nagel, Stein, and Wainger consider an
operator defined as in \eqref{defS} when $k = 2$ and the $\pi_j$ are
codimension $m$ submersions with $0 < m < n$.  It is rather simple to adapt the arguments in \cite{CNSW} to our circumstances to produce a short, relatively self-contained proof of Theorem~\ref{cnswml}.  We record this proof here both for the convenience of the reader and to provide some useful geometric intuition for later on.

\begin{proof}  Let $V$ be the vector space spanned by $\{X_w(0):w \in
W\}$, and let $r$ be its dimension.  Since each $X_j$ is nonzero, we
have that $r \geq 1$, and for the proof of the proposition, we may
assume that $r < n$.  Let $w_1,\ldots,w_r \in W$ be chosen so
$\{X_{w_1}(0), \ldots, X_{w_r}(0)\}$ is linearly independent, and
let $Y_j=X_{w_j}$.

For $t \in \reals^r$, define 
\[
t \cdot Y := \sum_{j=1}^r t_j Y_j, \quad \Gamma(t):= \exp(t \cdot Y)(0).
\]
Then $\Gamma$ has rank
$r$ at 0, so it is an embedding of a small neighborhood of 0 in
$\reals^r$ onto an $r$-dimensional submanifold $0 \in M \subset
\reals^n$.  Heuristically, the $X_j$ lie along $M$, so considering a
$\delta$-neighborhood $M_{\delta}$ of $M$, for each $1 \leq j \leq
k$, $\pi_j(M_{\delta})$ is of size $\delta^{n-r}$, which is
proportional to the size of $M_{\delta}$. Letting $\delta$ tend to
0, we obtain the requirement $\sum p_j^{-1} \leq 1$.
Unfortunately, this heuristic is misleading, since we only have
information about the Lie brackets of the $X_j$ at 0, and we will
have to use the techniques of \cite{CNSW} to obtain more
quantitative information; in the terminology of \cite{CNSW}, the submanifold $M$
will be invariant under the $X_j$ to infinite order at 0.

We will use a quantitative version of the Baker--Campbell--Hausdorff
formula, which is stated in \cite{CNSW} (for instance).  Let $V_1,\ldots,V_m$
and $W_1,\ldots,W_m$ be smooth vector fields on $\reals^n$.  Then
for $v,w \in \reals^m$, if we define $v \cdot V:= \sum v_jV_j$ and
$w \cdot W$ analogously, then
\begin{align*}
    \exp(v \cdot V) \exp(w \cdot W)(0) &= \exp(
    \sum_{k=1}^N c_k(v \cdot V,w \cdot W)) 
    + O(|v|^{N+1}+|w|^{N+1}).
\end{align*}
Here each $c_k$ is a homogeneous Lie polynomial of degree $k$.

For $s' \in \reals^r$ and $s'' \in \reals$ sufficiently close to 0
and $1 \leq j \leq k$, we define $G_j(s) = G_j(s',s'') =
\exp(s''X_j)\Gamma(s')$.  Since $\{X_w(0):w \in W\} \subset
\rm{span}\{Y_1,\ldots Y_r\}$, we may write
\[
    \sum_{k=1}^N c_k(s' \cdot Y, s'' X_j) = P_1^j(s) \cdot Y +
    \sum_{|\beta|=1}^N s^{\beta} W^j_{\beta,1},
\]
where $P_1^j$ is a vector-valued polynomial with $P_1^j(0) = 0$ and
each $W^j_{\beta,1}$ is in the span of $\{X_w:w \in W\}$ and
satisfies $W^j_{\beta,1}(0) = 0$.

We assume inductively that
\[
    G_j(s) = \exp(P_m^j(s) \cdot Y + \sum_{|\beta|=m}^N s^{\beta}
    W^j_{\beta,m})(0)+O(|s|^{N+1}),
\]
where $P_m^j$ is a vector-valued polynomial, $P_m^j(0)=0$, each
$W^j_{\beta,m}$ is in the span of $\{X_w:w \in W\}$, and
$W^j_{\beta,m}(0) = 0$.  The implicit constants in the $O(\cdot)$ notation are allowed to depend on $N$ (via the $C^N$ norms of the $X_j$).  Since $W^j_{\beta,m}(0) = 0$ for each $\beta, m$, 
\begin{align*}
    G_j(s) &= \exp(P_m^j(s) \cdot Y + \sum_{|\beta|=m}^N s^{\beta}
    W^j_{\beta,m}) \circ \exp(-\sum_{|\beta|=m}^N s^{\beta} W^j_{\beta,m})(0) +
    O(|s|^{N+1}).
 \end{align*}
 We apply the Baker--Campbell--Hausdorff formula and our inductive assumption $P_m^j(0) = 0$ to see that the right side equals the exponentiation of
  \begin{align*}
 & P_m^j(s) \cdot Y + \sum_{k=2}^N c_k(P_m^j(s) \cdot Y +
    \sum_{|\beta|=m}^N s^{\beta}W^j_{\beta,m},-\sum_{|\beta|=m}^N
    s^{\beta} W^j_{\beta,m}) \\
    &\qquad = P_m^j(s) \cdot Y + \sum_{|\beta|=m+1}^N
    s^{\beta}\tilde{W}^j_{\beta,m+1} + O(|s|^{N+1}).
 \end{align*}
Here the $\tilde{W}^j_{\beta,m+1}$ are in the span of $\{X_w:w \in W\}$ but do not necessarily vanish at zero.  We can add the ``error'' in the $\tilde{W}^j_{\beta,m+1}$ (the part that does not vanish at 0) to $P_m^j(s) \cdot Y$ and write
 \[
 G_j(s) = \exp(P_{m+1}^j(s) \cdot Y + \sum_{|\beta|=m+1}^N
    s^{\beta}W^j_{\beta,m+1})(0) + O(|s|^{N+1}),
\]
where $P_{m+1}^j$ is a vector-valued polynomial satisfying
$P_{m+1}^j(0) = 0$ and the $W^j_{\beta,m+1}$ are in the span of
$\{X_w:w \in W\}$ and satisfy $W^j_{\beta,m+1}(0) = 0$.  Proceeding
by induction,
\[
    G_j(s) = \exp(P_{N+1}^j(s) \cdot Y)+O(|s|^{N+1}),
\]
where $P_{N+1}^j(0)=0$, and we may assume that $P_{N+1}^j$ is of
degree less than $N+1$.

Consider coordinates $x = (x',x'') \in \reals^{r} \times
\reals^{n-r}$ on a neighborhood of 0 in $\reals^n$ so that $M =
\{x''=0\}$.  Define $\gamma_j(x,t):= e^{tX_j}(x)$.  We will also
write $\gamma_j = (\gamma_j',\gamma_j'')$ in the above coordinates.

Let $N$ be a fixed positive integer.  Since $\Gamma$ is a local
parametrization of $M$, and since $x''$ gives the distance from the
point $x$ to $M$, we have shown that for each $x',t$ sufficiently
close to 0 and each $1 \leq j \leq k$, 
\[
\gamma_j''(x',0'',t) = O(|x'|^{N+1}+|t|^{N+1}).
\]
More generally, since $\gamma_j(x,0) = x$,
\[
\gamma_j(x,t) = (x'+O(|t|),x''+O(|t||x''|) + O_N(|t|(|x'|^N+|t|^N))).
\]
Let
$\delta>0$.  Let $M_{\delta,N} = \{(x',x''):|x'|<\delta,
|x''|<C_N\delta^N\}$.

Then if $x \in M_{\delta,N}$, and $|t|<c\delta$, $\gamma_j(x,t) \in
M_{2\delta,N}$. Therefore $\delta|\pi_j(M_{\delta,N})| \lesssim
|M_{2\delta,N}|$, i.e., $|\pi_j(M_{\delta,N})| \lesssim
C_N\delta^{r-1}\delta^{N(n-r)}$. On the other hand, supposing $\sum
p_j^{-1} > 1$, we can find $N$ so that $r+N(n-r) < (\sum
p_j^{-1})((r-1)+N(n-r))$, and letting $\delta \to 0$ (with this $N$ fixed), $|M_{\delta,N}| \lesssim \prod
|\pi_j(M_{\delta,N})|^{1/p_j}$ cannot hold with a constant
independent of $\delta$.
\end{proof}

\section{Multi-parameter Carnot--Carath\'{e}odory Balls I}
\label{CCI}

We will spend much of Sections~\ref{CCI} and~\ref{CCII} reviewing properties of multi-parameter Carnot--Carath\'{e}odory balls which can be readily deduced from the work of Tao--Wright and Nagel--Stein--Wainger, though some new details will be needed.  An independent and far more in-depth discussion of these objects may be found in \cite{Street}.

If $\delta_1,\ldots,\delta_k$ are sufficiently small positive
numbers, and $x \in \reals^n$ is sufficiently close to 0, then we
define the multi-parameter Carnot--Carath\'{e}odory ball
$B(x;\delta_1,\ldots,\delta_k)$ to be the closure of the set of all
points
\begin{align*}
    e^{t_N \delta_{j_N}X_{j_N}}e^{t_{N-1} \delta_{j_{N-1}}
    X_{j_{N-1}}}\cdots e^{t_1 \delta_{j_1} X_{j_1}}(x),
 \end{align*}
 where $N \geq 0$ is an integer, $1 \leq j_i \leq k$, for each $1
    \leq i \leq N$, and $\sum_{i=1}^N|t_i| \leq 1$.

Essentially, this is the set of points which can be reached by
starting at $x$, then flowing first along one vector field
$\delta_{j_1}X_{j_1}$, then along another, and so on, for total time
less than or equal to 1.

We now give a heuristic proof of Theorems~\ref{nec} and~\ref{suff}.  We caution the reader that the equations in the following two paragraphs are not quite true or are false without additional assumptions, and require some proof at the very least.  

For Theorem~\ref{nec}, we consider $\Omega =
B(0;\delta_1,\ldots,\delta_k)$, with $0 < \delta_j \ll 1$.  Since
\[
    e^{t\delta_jX_j} \Omega \subset B(0;2\delta_1,\ldots,2\delta_k)
\]
whenever $|t| \leq 1$ and since
$$
    |B(x; 2\delta_1,\ldots,2\delta_k)| \sim |B(x;\delta_1,\ldots
    \delta_k)|,
$$
we have that $\frac{|\Omega|}{|\pi_j(\Omega)|} \gtrsim \delta_j$.  Thus if \eqref{wkineq3} holds, we must have that $\delta^b \lesssim |\Omega|$.  As 
$$
|B(x;\delta_1,\ldots,\delta_k)| \sim \sup_I\delta^{\deg(I)}|\lambda_I(x)|
$$
(roughly), we must have $b \in P$.

For Theorem~\ref{suff}, let $\Omega \subset \supp(a)$.  Then for $1 \leq j \leq k$, $\alpha_j$ represents the average of the 1-dimensional
Hausdorff measure of the set $\pi_j^{-1}\{y\}$ for $y \in
\pi_j(\Omega)$.  Hence we expect that for a `good' point $x \in
\Omega$ and for $1 \leq j \leq k$,
$$
    |\{t \in \reals:e^{tX_j}(x) \in \Omega\}| \sim \alpha_j.
$$
Iterating, $\Omega$ contains a set which `looks like' the
Carnot--Carath\'{e}odory ball $B(x;\alpha_1,\ldots,\alpha_k)$, so
whenever $\lambda_I(x) \neq 0$,
\begin{align} \label{lie6}
    \alpha^{\deg I} \sim \alpha^{\deg I}|\lambda_I(x)| \lesssim
    |B(x;\alpha_1,\ldots,\alpha_k)| \leq |\Omega|.
\end{align}
Thus, from the definition of $P$, we obtain \eqref{wkineq3}, and by
real interpolation, Theorem~\ref{suff}.

There are certain technicalities involved in the proof of Theorems~\ref{nec} and~\ref{suff}, for instance a correct analogue of
\eqref{lie6}, for which we will not be able to use the
Carnot--Carath\'{e}odory balls themselves.  The remainder of this
section will be devoted to translating Tao and Wright's discussion
of certain sets and mappings associated to the
Carnot--Carath\'{e}odory balls from the bilinear to the multilinear
setting.  We will return to the Carnot--Carath\'{e}odory balls
themselves in Section~\ref{CCII}.

We begin by reviewing some notation from \cite{TW}.  For the remainder of this section $\eps > 0$ will be a small
parameter, and $K$ will be a large parameter.  We will be more
specific about these quantities later on.  For the purposes of this
section, all implicit constants depend on $\eps$ and on the $\pi_j$.

We let $\delta_1,\ldots,\delta_k$ be positive numbers which satisfy the smallness and the nondegeneracy conditions
\begin{align} \label{small}
    \delta_j \leq c_{\eps,K}, \quad 1 \leq j \leq k\\
 \label{nondeg}
    \delta_i \lesssim \delta_j^{\eps}, \quad 1 \leq i,j \leq k.
\end{align}
We remark that the nondegeneracy condition is necessary for the balls
$B(x;\delta_1,\ldots,\delta_k)$ to satisfy the doubling property of \cite{NSW} and \cite[Ch.~1]{St}, for instance. Indeed, there is an example in \cite{ChRL} of a pair of vector fields
$X_1$ and $X_2$ which satisfy the H\"{o}rmander condition but have the property that there is no universal constant $C$ such that
\[
    |B(0;2\delta_1,2\delta_2)| \leq C|B(0;\delta_1,\delta_2)|
\]
for all sufficiently small $\delta_1,\delta_2$.

By Theorem~\ref{cnswml}, we may assume that there exists $I_0 \in
W^n$ with $\lambda_{I_0}(0) \neq 0$.  Shrinking $V$, we may in fact assume that $|\lambda_{I_0}| \sim 1$ throughout $V$.  We let $d:= \sum_{j=1}^k (\deg I_0)_j$, and define $\mathbf{I}$ to be the finite set
\begin{align} \label{bfI}
    \bfI := \{I \in W^n:(\deg I)_j \leq \tfrac{d}{\eps}, 1 \leq j
    \leq k\}.
\end{align}
If $I \notin \bfI$ and $\delta = (\delta_1,\ldots,\delta_k)$ is as above, then $(K\delta)^{\deg I} \lesssim (K\delta)^{\deg I_0}$.  We define $\Lambda = \Lambda_{K\delta}$ by
\begin{align} \label{defLambda}
    \Lambda(x) = ((K\delta)^{\deg I} \lambda_I(x))_{I \in \bfI}.
\end{align}
Since $|\Lambda(x)| \gtrsim (K\delta)^{\deg I_0}$, we have that
\begin{align} \label{lbLambda}
    (K\delta)^{\deg I}|\lambda_I(x)| \lesssim C_I|\Lambda(x)|,
\end{align}
for all $I \in W^n$ and $x \in V$.  

With $x_0 \in V$ fixed, we choose $I_{x_0} = (w_1,\ldots,w_n) \in \bfI$ so that
\begin{align} \label{Ix0}
    (K\delta)^{\deg I_{x_0}}|\lambda_{I_{x_0}}(x_0)| \sim
    |\Lambda(x_0)|.
\end{align}
Our goal is to gain a basic understanding of the mapping $\Phi = \Phi_{x_0,K\delta}$ defined by 
\begin{align} \label{defPhi}
    \Phi(t_1,\ldots,t_n) = \exp(\sum_{j=1}^n K^{-1}(K\delta)^{\deg
    w_j}t_jX_{w_j})(x_0)
\end{align}
for $t$ near 0 in $\reals^n$.  By smallness of the $K\delta_j$, the domain of definition of $\Phi$ may be taken to be uniform in $x_0\in V$ and $\delta$.

Since 
\begin{align} \label{detDPhi}
    |\det D\Phi(0)| = K^{-n}(K\delta)^{\deg
    I}|\lambda_{I_{x_0}}(x_0)| \sim |\Lambda(x_0)| \neq 0,
\end{align}
the mapping $\Phi$ is a diffeomorphism on a neighborhood $U$ of 0.  For $t \in U$ and $w \in W$, we define $Y_w$ to be the pullback by $\Phi$ of $K^{-1}(K\delta)^{\deg w}X_w$:
\begin{align} \label{defYw}
    Y_w(t) = (D\Phi(t))^{-1}[K^{-1}(K\delta)^{\deg w} X_w(\Phi(t))].
\end{align}
Although we are suppressing this in the notation, we keep in mind throughout that the map $\Phi$, the set $U$, and the vector fields $Y_w$ depend on the base point $x_0 \in V$ as well as $\delta$.

The following lemmas are proved in \cite{TW}.  Though the authors
only claim the results in the case $k = 2$, their proofs extend to
the multilinear case with almost no alteration.  For what follows,
all bounds are uniform in $K > C_{\eps}$,
$0<\delta_1,\ldots,\delta_k < c_{K,\eps}$ satisfying
\eqref{nondeg}, and the base point $x_0 \in V$, but may depend on $\eps$ and the $X_j$.

\begin{lemma} \label{L:1} If $B_r(0) \subset U$, for some $0<r \lesssim 1$, then
\begin{align} \label{Yisimdi}
    Y_{w_i}(t) = \partial_i + O(\tfrac{|t|}{K})
\end{align}
and in particular,
\begin{align} \label{detYi}
    |\det(Y_{w_1},\ldots,Y_{w_n})(t)| \sim 1
\end{align}
for all $t \in B_r(0)$.
\end{lemma}

\begin{lemma} \label{L:2}
    If $B_r(0) \subset U$ for some $0 < r \lesssim 1$, then
    \begin{align} \label{boundLambda}
        |\Lambda \circ \Phi(t)| \sim (K\delta)^{\deg
        I_{x_0}}|\lambda_{I_{x_0}} \circ \Phi(t)|
    \end{align}
    on $B_r(0)$.
\end{lemma}

\begin{lemma} \label{domainPhi}
    There exists $C \sim 1$ so that $B:= B_C(0) \subset U$.
\end{lemma}

\begin{lemma} \label{L:4}
    If $w \in W$, then
    \begin{align} \label{YwCM}
        \|Y_w\|_{C^M(B)} \leq C_{w,M},
    \end{align}
    provided $K$ is sufficiently large depending on $w$, $M$, and $\eps$.
\end{lemma}

\begin{lemma} \label{L:5}
    If $E$ is a measurable subset of $B$, then
    \begin{align} \label{PhiE}
        |\Phi(E)| \sim K^{-n}|\Lambda(x_0)||E|.
    \end{align}
\end{lemma}

We will not repeat the proofs of these lemmas, for which we direct
the reader to Proposition 4.1 of \cite{TW}.  To offer some
explanation for the parameter $K$, however, we will sketch the
argument for the first three lemmas.

It is easy to compute $Y_{w_j}(0) = \partial_j$, and moreover, at each point $t \in \reals^n$, we have
\[
    r\partial_r:= \sum_{j=1}^n t_j \partial_j = \sum_{j=1}^n t_j
    Y_{w_j}.
\]
With these facts, and after some algebra and differential identities, we can compute
radial derivatives using Lie brackets with the $Y_{w_j}$.  The factor $K^{-1}$ in the definitions above helps minimize the contribution coming from higher order Lie brackets once we pull back (which has the effect of replacing $\delta$ with 1), since 
\begin{align} \label{YwYw'}
    [Y_w,Y_{w'}] = K^{-1} (D\Phi(t))^{-1}(K^{-1}(K\delta)^{\deg w+\deg w'} [X_w,X_{w'}](\Phi(t)))
\end{align}
for any $w,w' \in W$.  Together with Gronwall's inequality, the
bounds on the radial derivatives imply the first two lemmas.  The
third lemma follows from the first two and continuity, since
\[
    |\det D\Phi(t)| = \frac{K^{-n}(K\delta)^{\deg
    I_{x_0}}|\lambda_{I_{x_0}}(\Phi(t))|}{|\det(Y_{w_1},\ldots,
    Y_{w_n})(t)|}.
\]
We note that \eqref{YwYw'} would hold if we replaced each instance
of $K\delta$ with $\delta$ in the above discussion, but if we did
that, unless each of the $w_i$ was actually in $\{1,\ldots,k\}$,
$\Phi(B_1(0))$ would be much smaller than
$B(x_0;\delta_1,\ldots,\delta_k)$ for large $K$.

\section{Proof of Theorem~\ref{nec}} \label{necpf}

Here we adapt Tao and Wright's proof that the bound \eqref{wkineq2}
can only hold if $b(p)$ (see \eqref{defb}) lies in the Newton
polytope $P$.

We first record some geometric properties of $P$.  By definition,
$P$ is convex and
\begin{align} \label{Pinfinite}
    \text{$b \in P$ and $b'\geq b$ implies $b' \in P$.}
\end{align}
Moreover, by Theorem~\ref{cnswml}, we may assume that $P \neq
\emptyset$.  Since the vertices of $P$ are $k$-tuples of
non-negative integers, by \eqref{Pinfinite} $P$ has only finitely
many vertices.  It is also clear from the definition that $b \in P$
implies that $\sum_{j=1}^k b_j \geq n$.

Now suppose that $p$ satisfies \eqref{p>1} and that $b(p)$ does
not lie in $P$.  Then there exists $a \in \reals^k$ and $d \in
\reals$ so that
\[
    b(p) \in \{b \in \reals^k:a \cdot b < d \} =: H_-, \quad 
    P \subset \{b \in \reals^k:a \cdot b > d\} =: H_+.
\]
From \eqref{Pinfinite}, each entry of $a$ is non-negative, and since
$b(p) \in [0,\infty)^k$, $d>0$.  Since $P$ has finitely many vertices,
and since each vertex lying in $H_+$ implies that $P$ lies in $H_+$,
we may assume that each entry of $a$ is positive by continuity.
Finally, by scaling, we may assume that $d=1$.

Let $\delta_0 > 0$.  Then since $a \in (0,\infty)^k$,
\begin{align} \label{defdelta}
    \delta:=(\delta_0^{a_1},\ldots,\delta_0^{a_k})
\end{align}
satisfies \eqref{nondeg} for some $\eps > 0$, independent of
$\delta_0$.  Shrinking $\eps$ if needed, we may assume that the set
$\bfI$ defined in \eqref{bfI} contains all of the vertices of $P$
and that
\begin{align} \label{adotb}
    a \cdot b(p) < 1-\eps  \qquad     a \cdot b > 1+\eps
\end{align}
for each $b \in P$.

We will use the results of Section~\ref{CCI} to show that $S$ is not of
restricted weak-type $(p_1,\ldots,p_k)$.

With $x_0 = 0$, choose $K$ large enough that \eqref{boundLambda} and
\eqref{PhiE} hold and \eqref{YwCM} holds with $w=1,\ldots,k$ and
$M=0$ on $B=B_C(0)$, whenever $\delta$ satisfies \eqref{small} and
\eqref{nondeg}.

Let $0 < c_{\eps} <C$ be sufficiently small for later purposes, and
let
\[
    \Omega = \Phi(B_{c_{\eps}}(0)).
\]
If $x = \Phi(t) \in \Omega$, $s \ll 1$, and $1 \leq j \leq k$, then
\[
    e^{s \delta_j X_j}(x) = \Phi(e^{sY_j}(t)) \subset \Phi(B_C(0))
\]
by \eqref{YwCM} and the smallness of $c_{\eps}$.  Hence
\[
    \frac{|\Omega|}{|\pi_j(\Omega)|} \gtrsim \delta_j.
\]
Thus \eqref{wkineq3} implies that
\[
    \delta^{b(p)} \lesssim |\Omega|.
\]
But by \eqref{PhiE},
\[
    |\Omega| \lesssim K^{-n}|\Lambda(0)| \lesssim C_K\delta^{b_0}
\]
for some $b_0 \in P$.  Hence (as $K$ depends on $\eps$), by \eqref{adotb} we have
\[
    \delta_0^{1-\eps} < \delta_0^{b(p) \cdot a} \lesssim C_K
    \delta_0^{b_0 \cdot a} \lesssim C_{K}\delta_0^{1+\eps}.
\]
Letting $\delta_0 \to 0$, we obtain a contradiction.

\section{Multi-parameter Carnot--Carath\'{e}odory Balls II}
\label{CCII}

In this section, we return to the multi-parameter
Carnot--Carath\'{e}odory balls $B(x;\delta_1,\ldots,\delta_k)$,
defined in Section~\ref{CCI}.  We will obtain estimates for the volumes
of these balls, and will use these estimates to give alternate
statements to Theorems~\ref{nec} and~\ref{suff}.  We will also prove
that, under the assumptions \eqref{small} and \eqref{nondeg}, the balls satisfy the doubling
property.  As mentioned earlier, most of the needed results can be obtained by translating existing results from the bilinear to the multilinear setting.

As before, we let $V$ be a small neighborhood of 0, and let
$X_1,\ldots,X_k$ be smooth vector fields defined on $\reals^n$.  We
assume that there exists $I_0 \in W^n$ so that $|\lambda_{I_0}(x)|
\sim 1$ for all $x \in V$.  All implicit constants in this section
depend on $\eps > 0$ and the vector fields $X_1,\ldots, X_k$.

If $\delta = (\delta_1,\ldots,\delta_k)$ is a $k$-tuple of positive
numbers satisfying \eqref{nondeg}, we let $\Lambda_{\delta}$ be
defined as in \eqref{defLambda}, with $K = 1$.

\begin{prop}
    There exists a constant $C_{\eps}>1$ such that
    \[
        |B(x;\delta)| \lesssim C_{\eps}|\Lambda_{\delta}(x)|,
    \]
    whenever $x \in V$ and $\delta = (\delta_1,\ldots,\delta_k)$ is a $k$-tuple of
    positive numbers satisfying \eqref{nondeg} and $\delta_i <
    C_{\eps}^{-1}$.
\end{prop}

\begin{proof}  Let $K=K_{\eps}$ be sufficiently large that the conclusions of Lemmas~\ref{L:1}--\ref{L:5} hold (with $M=0$ in Lemma~\ref{L:4}), uniformly in $x \in V$ and $\delta$ satisfying \eqref{small} and \eqref{nondeg}.  (Note that $c_{K,\eps}$ is indirectly a function of $\eps$ alone.)  Henceforth $x$, $\delta$ will be fixed, with $\Phi$ and the $Y_w$ defined accordingly.

Define $\tilde{B}(0;c_{\eps},\ldots,c_{\eps})$ to be the 
closure of the set of all points
\begin{align*}
    e^{t_NY_{j_N}} \cdots e^{t_1 Y_{j_1}}(0),
\end{align*}
where $N$ is a non-negative integer, $1 \leq j_i \leq k$ for $1 \leq i \leq N$ and $\sum_{i=1}^N|t_i| \leq c_{\eps}$.  

Then, if $c_{\eps}$ is sufficiently small depending on the
$C^0$ norms of $Y_1,\ldots,Y_k$, 
we have that $\tilde{B}(0;c_{\eps},\ldots,c_{\eps})$
is contained in $B$ (the domain of $\Phi$).  Since $Y_i$ is the
pullback by $\Phi$ of $\delta_i X_i$,
\[
    \tilde{B}(0;c_{\eps},\ldots,c_{\eps}) = \Phi^{-1}(B(x;
    c_{\eps}\delta_1,\ldots,c_{\eps}\delta_k)).
\]
Thus, by \eqref{PhiE},
\[
    |B(x;c_{\eps}\delta_1,\ldots,c_{\eps}\delta_k)| \lesssim
    K^{-n}|\Lambda_{K\delta}(x)| \cdot |B| \sim
    C_{\eps}|\Lambda_{\delta}(x)|.
\]
\end{proof}

\begin{prop}
   There exists
    $C_{\eps}>1$ such that whenever $x \in V$ and $\delta$ obeys \eqref{nondeg} and $\delta_i \leq C_{\eps}^{-1}$, we have
    \[
        |B(x;\delta)| \geq C_{\eps}^{-1}|\Lambda_{\delta}(x)|.
    \]
    Furthermore, there exists a sequence $\bfj = (j_1,\ldots,j_n)
    \in \{1,\ldots,k\}^n$ so that
    \begin{align} \label{lbBj}
        |B_{\bfj}(x;\delta)| \geq
        C_{\eps}^{-1}|\Lambda_{\delta}(x)|,
    \end{align}
    where $B_{\bfj}(x;\delta)$ is defined to be the closure of the
    set
    \[
        \{e^{t_n \delta_{j_n}X_{j_n}} \cdots e^{t_1 \delta_{j_1}
        X_{j_1}}(x):  |t_i| \leq  1, 1 \leq i \leq n\}.
    \]
\end{prop}

The proof of this proposition, which uses the Arzela--Ascoli theorem and the Nagel--Stein--Wainger theory, is based on the proof of a related lemma in \cite{ChRL}.  The proof of the analogous fact in \cite{TW} seems more specialized to the bilinear case, as it uses the fact that there are only two possibilities for $\bfj$, namely $(1,2,1,\ldots)$ and $(2,1,2,\ldots)$.

\begin{proof}  If the statements in the proposition were false, there would exist sequences $\{x^{(\ell)}\}$ of points in $V$ and $\{\delta^{(\ell)}\}$ of $k$-tuples satisfying the hypotheses of the proposition such that
$$
\lim_{\ell \to \infty} \delta^{(\ell)} = (0,\ldots,0)
$$
and if
$$
\Phi^{(\ell)} := \Phi_{x^{(\ell)},K\delta^{(\ell)}},
$$
then
\begin{align} \label{ineq:limit_0}
\lim_{\ell \to \infty}    |[\Phi^{(\ell)}]^{-1}(B_{\bfj}(x;\delta^{(\ell)}))| = 0
\end{align}
for each $\bfj \in \{1,\ldots,k\}^n$.  (Here $K = K_{\eps}$ is sufficiently large to allow the applications of Lemmas~\ref{L:1}--\ref{L:5} below.)

Letting $Y^{(\ell)}_i$ be the pullback of $\delta_i^{(\ell)}X_i$ by
$\Phi^{(\ell)}$, by \eqref{YwCM} and the
Arzela--Ascoli theorem, passing to subsequences if necessary, there exists a vector field $Y_i$ so that 
$Y_i^{(\ell)} \to Y_i$ in $C^M(B)$, with $B$ as in Lemma~\ref{domainPhi} and for $M$ arbitrarily large.  We take $M \gg \eps^{-1}$ so that each sequence $Y_w^{(\ell)}$ with $w \in \bfI$ (defined in \eqref{bfI}) converges in $C^N(B)$, for $N$ large.  By Lemma~\ref{L:1}, for each $\ell$ there exists $(w_1^{(\ell)},\ldots,w_n^{(\ell)}) \in \bfI$ so that 
$$
|\det(Y_{w_1^{(\ell)}}^{(\ell)},\ldots, Y_{w_n^{(\ell)}}^{(\ell)})| \sim 1
$$
in $B$.  By finiteness of $\bfI$, after passing to a subsequence we may assume that there is a single such $n$-tuple, $(w_1,\ldots,w_n)$.  We then have that
$$
|\det(Y_{w_1},\ldots,Y_{w_n})| \sim 1
$$
in $B$.

Now we have a contradiction.  On the one hand, by the work of Nagel, Stein, and Wainger in \cite{NSW}, there exists a sequence $\bfj$ and a constant $0 < c_{\eps} < 1$ so that 
\begin{align} \label{ineq:lb_limit_ball}
    |\tilde{B}_{\bfj}(0,c_{\eps},\ldots,c_{\eps})| \sim 1,
\end{align}
where $\tilde{B}_{\bfj}(0,c_{\eps},\ldots,c_{\eps})$ is the closure of the set
$$
\{e^{t_nc_{\eps}Y_{j_n}}\cdots e^{t_1c_{\eps}Y_{j_1}}(0) : |t_i| \leq 1,\: 1 \leq i \leq n\}.
$$
Thus, if we let 
\[
\Phi_{\bfj}(t) := e^{t_n Y_{j_n}}\cdots e^{t_1 Y_{j_1}}(0),
\]
then for some $t^0$ with $\sum|t^0_i| < c_{\eps}$, $\det D\Phi_{\bfj}(t_0) \neq 0$.

On the other hand, because the vector fields $Y_i^{(\ell)}$ converge to $Y_i$ in $C^M(B)$ for large $M$, the maps $\Phi_{\bfj}^{(\ell)}$ (defined analogously to the map above) converge to $\Phi_{\bfj}$ in (say) $C^2(B)$.  Thus $|\det D\Phi_{\bfj}^{(\ell)}(t_0)| > c > 0$ (eventually), and so the $\Phi_{\bfj}^{(\ell)}$ are injective on a uniform neighborhood of $t_0$.  This gives a lower bound on 
\[
|\tilde{B}_{\bfj}^{(\ell)}(0;c_{\eps})| := \Phi_{\bfj}(\{t \in \reals^n :  |t_i| \leq c_{\eps}, 1 \leq i \leq n\}).
\]
But the above set is just $\Phi^{-1}_{K\delta^{(\ell)}}(B_{\bfj}(x;\delta^{(\ell)}))$, and we have the promised contradiction.
\end{proof}

The two propositions imply the following doubling property.
\begin{cor}
    Whenever $\delta = (\delta_1,\ldots,\delta_k)$ satisfies \eqref{small} and
    \eqref{nondeg}, we
    have that
    \[
        |B(x;2\delta_1,\ldots,2\delta_k)| \leq C_{\eps}
        |B(x;\delta_1,\ldots,\delta_k)|
    \]
    uniformly in $x \in V$.
\end{cor}

We also obtain the following alternative characterization of the
Newton polytope $P$.

\begin{prop}
    The Newton polytope $P$ associated to the vector fields
    $X_1,\ldots,X_k$ is equal to the set of all points $b =
    (b_1,\ldots,b_k)$ so that
    \begin{align} \label{Bdeltab}
        |B(0;\delta_1,\ldots,\delta_k)| \gtrsim \delta^b,
    \end{align}
    where the implicit constant depends on $\eps$, but is uniform in
    $\delta_1,\ldots,\delta_k > 0$ satisfying \eqref{small} and \eqref{nondeg}.
\end{prop}

\begin{proof}  If $b \notin P$, then by the propositions of this
section and the proof of Theorem~\ref{nec}, for some $\eps > 0$ no
uniform bound
\[
    |B(0;\delta_1,\ldots,\delta_k)| \gtrsim \prod_{j=1}^k
    \delta_j^{b_j}
\]
can hold.  If $b \in P$ and $\delta_1,\ldots,\delta_k$ are any
positive numbers, then
\[
    \delta^b \leq \sum_{I \in {\bfI}_0}\delta^I \sim
    |\Lambda_{\delta}(0)|,
\]
where ${\bfI}_0$ is the set of vertices of $P$.  By the propositions
of this section, we thus have \eqref{Bdeltab}. 
\end{proof}

By the doubling property, we have 
\[
    |B(x;\delta_1,\ldots,\delta_i,\ldots,\delta_k)| \sim
    |B(x;\delta_1,\ldots,2\delta_i,\ldots,\delta_k)|.
\]
Since $e^{t\delta_iX_i}B(x;\delta) \subset B(x;\delta_1,\ldots,2\delta_i,\ldots,\delta_k)$, while $\pi_i \circ e^{tX_j} \equiv \pi_i$, by the coarea formula we must have
\[
    \frac{|B(x;\delta_1,\ldots,\delta_k)|}{|\pi_i(
    B(x;\delta_1,\ldots,\delta_k))|} \gtrsim_{\eps} \delta_i.
\]
From this and the propositions, we are able to obtain geometric
versions of Theorems~\ref{nec} and~\ref{suff}.

Theorem~\ref{nec} is thus equivalent to the following tautology: If $S$
is of restricted weak-type $(p_1,\ldots,p_k)$, then \eqref{wkineq3} (with the implicit constant depending on $\eps$) 
holds whenever $\delta$ satisfies \eqref{small} and \eqref{nondeg} and $\Omega = B(0;\delta_1,\ldots,\delta_k)$.

We now give an alternative, more geometric, statement of Theorem~\ref{suff}, analogous to the formulation in \cite{ChRL}.

\begin{theorem} \label{thm:geometric}
    Assume that $p = (p_1,\ldots,p_k)$ satisfies \eqref{p>1}.
    Suppose that for each $\eps>0$ there exists $c_{\eps}>1$ so that
    whenever $\delta = (\delta_1,\ldots,\delta_k)$ is a $k$-tuple of
    sufficiently small (depending on $\eps$) positive numbers
    satisfying \eqref{nondeg}, we have
    \[
        |B(0;\delta)| \leq c_{\eps} \prod_{j=1}^k
        |\pi_j(B(0;\delta))|^{1/p_j}.
    \]
    Then whenever $\tilde{p}>p$ (i.e. $\tilde{p}_i > p_i$, $1 \leq i
    \leq k$), we have \eqref{stineq}.
\end{theorem}

\section{Proof of Theorem~\ref{suff}}

In this section we will prove that if $V$ is a sufficiently small
neighborhood of 0 and $a \in C^{\infty}_c(V)$, then $S$ is of
restricted weak-type $(p_1,\ldots,p_k)$ whenever $p$ satisfies
\eqref{p>1} and $b(p)$ lies in the interior of the Newton polytope
$P$.  One may use the arguments in this section together with a
partition of unity to see the following:  If $V$ is bounded and for
every vertex $b$ of $P$ and $x \in V$ there exists $I \in W^n$ with
$\deg(I) \leq b$ and $\lambda_I(x) \neq 0$, then $V$ is sufficiently
small in the above sense.  By real interpolation, this proves
Theorem~\ref{suff}.

We note that the arguments of this section are largely based on those in \cite{ChRL} and \cite{TW}, but some new details, such as in the refinement, are needed in the multilinear setting.

We let ${\bfI}_0 \subset W^n$ be the (finite) set of all $n$-tuples
of words $I$ such that $\lambda_I(0) \neq 0$ and such that $\deg(I)$
is a vertex of $P$.  By passing to a smaller subset of $V$ if
needed, we may assume that if $I \in {\bfI}_0$, then $\lambda_I \sim
1$ on $V$.

Let $\Omega \subset V$ be a Borel set having positive Lebesgue
measure, and let $\alpha_1,\ldots,\alpha_k$ be defined as in
\eqref{defalpha}.  By symmetry, we may assume that
\[
    \alpha_1 \geq \ldots \geq \alpha_k.
\]
Since $\alpha_j \lesssim \diam(V)$, by passing to a smaller subset
of $V$ if needed, we may assume that each $\alpha_j$ is as small as
we like.

In order to prove that
\begin{align} \tag{\ref{wkineq3}}
    |\Omega| \gtrsim \alpha_1^{b_1} \cdots \alpha_k^{b_k}
\end{align}
for $b$ lying in the interior of $P$, it suffices to show that if $b
\in P$, there exists a constant $C>0$ so that for every $\eps>0$ we
have
\begin{align} \label{wkineq4}
    |\Omega| \gtrsim \alpha_k^{C\eps}\alpha^b,
\end{align}
where in the preceding statement and for the remainder of this
section the implicit constant is allowed to depend on $\eps$.  To
see that this suffices, note that if $b \in \rm{int}\,P$, then there exists
$b' \in \rm{int}\,P$ such that $b' < b$.  Then $b > b'+(0,\ldots,0,C\eps)$ if
$\eps$ is sufficiently small, so \eqref{wkineq4} with $b = b'$
implies \eqref{wkineq3}, by smallness of the $\alpha_j$.

We will assume throughout that $\eps$ is small enough that ${\bfI}_0
\subset \bfI$, where $\bfI$ is the set defined in \eqref{bfI}.

\subsection{Refining $\Omega$}  To apply what we learned in previous sections, we must put ourselves in the situation of considering a ``large'' subset of $\Omega$ which ``looks like'' a Carnot Carath\'eodory ball with weakly comparable radii.  In this subsection, we will make an initial refinement of $\Omega$ which will give us the weakly-comparable ``radii''.

By boundedness of $V$, we may decompose
$\Omega$ as the disjoint union of $\lesssim \alpha_k^{-C\eps}$ Borel
sets of diameter $\lesssim \alpha_k^{\eps}$.  Henceforth, we will
work with the largest of these, denoted $\tilde{\Omega}$, which has
measure
\[
    |\tilde{\Omega}| \gtrsim \alpha_k^{C\eps}|\Omega|.
\]

Next, we refine $\tilde{\Omega}$.  Let $\bfJ \in
\{1,\ldots,k\}^{nk^n}$ be a sequence which is formed by
concatenating all of the elements of $\{1,\ldots,k\}^n$ in some
order.  Our next goal is to construct a sequence of refinements
\[
    \Omega_0 \subset \Omega_1 \subset \cdots \subset \Omega_{nk^n}
    \subset \tilde{\Omega}
\]
of $\tilde{\Omega}$ so that $\Omega_0 \neq \emptyset$ and so that
for $1 \leq i \leq nk^n$ and $x \in \Omega_{i-1}$, the set
\[
    \{t :|t| \ll 1 \, \text{ and } \,
    e^{tX_{J_i}}(x) \in \Omega_{i}\}
\]
has a particular form.

The following definition is due to Tao and Wright.

\begin{defn} If $0 < \eps,w \ll 1$, then a {\it central set of width
$w$} is a subset $S$ of $[-w,w]$ having positive measure and such
that for any interval $I$
\[
    |I \cap S| \lesssim \left(\tfrac{|I|}{w}\right)^{\eps}|S|.
\]
\end{defn}
There is an analogous definition due to Christ in \cite{ChRL}.

\begin{lemma} \label{TW8.2}
    Let $\Omega' \subset \tilde{\Omega}$ with $|\Omega'| \gtrsim
    \alpha_k^{C\eps}|\tilde{\Omega}|$.  Then if $1 \leq j \leq k$,
    there exists a subset $\langle \Omega' \rangle_j \subset
    \Omega'$ so that $|\langle \Omega' \rangle_j| \gtrsim
    \alpha_k^{C'\eps}|\tilde{\Omega}|$ and so that for each $x \in
    \langle \Omega' \rangle_j$,
    \[
        \{t:|t| \ll 1 \, \text{ and } \,
        e^{tX_j}(x) \in \Omega'\}
    \]
    is a central set of width $w$,
    \[
        \alpha_k^{C'\eps} \alpha_j \lesssim w \lesssim
        \alpha_k^{\eps}
    \]
    and measure $\gtrsim \alpha_k^{C'\eps}\alpha_j$.  Here $C'$ is a
    constant which is larger than $C$, but independent of $\eps$.
\end{lemma}
The proof is the same as that of Lemma~8.2 in \cite{TW} and will be omitted.  

We now define the refinements as follows:
\[
    \Omega_{nk^n}:=\langle \tilde{\Omega} \rangle_{J_{nk^n}},
\]
given $\Omega_i$, $2 \leq i \leq nk^n$, 
\[
    \Omega_{i-1} := \langle \Omega_i \rangle_{J_{i-1}},
\]
and $\Omega_0:=\Omega_1$.  Then $|\Omega_0| \gtrsim
\alpha_k^{C\eps}|\Omega| > 0$, so $\Omega_0 \neq \emptyset$.
Moreover, since $\Omega_i$ satisfies the conclusion of Lemma~\ref{TW8.2} with $j = J_i$, and since $\Omega_{i-1} \subset
\Omega_i$, whenever $x \in \Omega_{i-1}$,
\[
    \{t:|t| \ll 1 \, \text{ and } \,
    e^{tX_{J_i}}(x) \in \Omega_i\}
\]
is a central set of width $w_i$, $C_{\eps}^{-1}\alpha_k^{C\eps}
\alpha_{J_i} \leq w_i \lesssim \alpha_k^{\eps}$ and measure $\geq
C_{\eps}^{-1} \alpha_k^{C\eps} \alpha_{J_i}$, where $C_{\eps}$ is a
constant depending on $\eps$.

\subsection {Filling out $\tilde{\Omega}$}  In this subsection, we show that $\tilde{\Omega}$ contains a set which looks like a Carnot--Carath\'eodory ball with radii coming from the measures of the central sets in the previous subsection.  We also sketch a heuristic argument for the conclusion of the proof.

Fix a base point $x_0 \in \Omega_0$ and set
\[
    \delta_j:= C_{\eps}^{-1}\alpha_k^{C\eps} \alpha_j, \hspace{.5cm}
    1 \leq j \leq k.
\]
Then the $\delta_j$ satisfy \eqref{nondeg}, though possibly with a
smaller value of $\eps$.  Hence we may choose an $n$-tuple $\bfj \in
\{1,\ldots,k\}^n$ so that
\[
    |B_{\bfj}(x_0;\delta_1,\ldots,\delta_k)| \sim
    |B(x_0;\delta_1,\ldots,\delta_k)|.
\]
For $1 \leq i \leq n$, define on a ball $B_i$ centered at 0 in
$\reals^i$ of radius $\sim 1$
\[
    \Phi_{\bfj}^i(t_1,\ldots,t_i):= e^{t_iX_{j_i}}\cdots
    e^{t_1X_{j_1}}(x_0).
\]

Let $\ell$ be such that $(J_{\ell n+1},\ldots,J_{\ell n+n}) = (j_1,\ldots,j_n)$.  Define
\[
    T_1 = \{t_1 \in \reals:|t_1| \ll 1 \, \text{ and }\, \Phi_{\bfj}^1(t_1) \in \Omega_{\ell n+1}\}.
\]
Then since $x_0 \in \Omega_{\ell n+1}$, $T_1$ is a central set of width
$w_1$ (after reindexing), with
\[
    \alpha_k^{C\eps}\alpha_{j_1} \lesssim w_1 \lesssim
    \alpha_k^{\eps}
\]
and measure $\gtrsim \alpha_k^{C\eps}\alpha_{j_1}$.  Assuming
$T_{i-1}$ has been defined, and $2 \leq i \leq n$, we define
\[
    \tau_i(t):= \{t_i \in \reals:|t_i| \ll 1 \,\text{ and }\,
     \Phi_{\bfj}^i(t,t_i) \in \Omega_{\ell n+i}\}
\]
whenever $t \in T_{i-1}$, and let
\[
    T_i:= \{(t,t_i) \in \reals^i:t \in T_{i-1} \,
    \text{ and }\,   t_i \in \tau_i(t)\}.
\]
Then each of the $\tau_i(t)$ is a central set of width $w_i$,
\[
    \alpha_k^{C\eps} \alpha_{j_i} \lesssim w_i \lesssim
    \alpha_k^{\eps}
\]
and measure $\gtrsim \alpha_k^{C\eps} \alpha_{j_i}$.

Since $\Phi_{\bfj}^n(T_n) \subset \tilde{\Omega}$, it suffices to
prove a lower bound for $|\Phi_{\bfj}^n(T_n)|$.  In fact, since
\begin{align*}
    |B_{\bfj}(x_0;\delta_1,\ldots,\delta_k)| \sim     |B(x_0;\delta_1,\ldots,\delta_k)| \gtrsim \sum_{I \in {\bfI}_0} \delta^{\deg(I)} \gtrsim \alpha_k^{C\eps}\alpha^b,
\end{align*}
it suffices to show that
\[
    |\Phi_{\bfj}^n(T_n)| \gtrsim
    \alpha_k^{C\eps}|B_{\bfj}(x_0;\delta_1,\ldots,\delta_k)|.
\]

The rest of this section will be devoted to making the following
heuristic argument rigorous:

Let
\[
    \tilde{T_n} = \{(t_1,\ldots,t_n) \in \reals^n:|t_i| \leq
    \delta_{j_i}\}.
\]
Then
\begin{align*}
    |B_{\bfj}(x_0;\delta_1,\ldots,\delta_k)| &= |\Phi_{\bfj}^n(\tilde{T_n})| \sim \int_{\tilde{T_n}}|\det \partial_t \Phi_{\bfj}^n(t)|dt \\  &\lesssim \tfrac{|\tilde{T_n}|}{|T_n|} \int_{T_n} |\det \partial_t
    \Phi_{\bfj}^n(t)|dt 
     \lesssim  \alpha_k^{-C\eps}|\Phi_{\bfj}^n(T_n)|.
\end{align*}
In the lines above, we certainly ignored some details, but despite this, the properties of
central sets, together with the smoothness of the $X_j$ make this
heuristic surprisingly close to the truth.

\subsection{Polynomials on $T_n$}

This subsection closely follows the work of Christ in \cite{ChRL}.

\begin{lemma}
    If $S$ is a central set of width $w \lesssim 1$, and $P$ is a
    polynomial of degree $D$ on $\reals$, then
    \[
        |P| \gtrsim C_D \|P\|_{L^{\infty}([-w,w])}
    \]
    on a subset $S' \subset S$ of measure $\gtrsim |S|$.
\end{lemma}

\begin{proof}[Sketch of proof.]  This is proved in \cite{ChRL}.  Roughly, the
values of $x \in [-w,w]$ such that $|P(x)|
\ll\|P\|_{L^{\infty}([-w,w])}$ lie near the $\leq D$ complex zeros
of $P$ a distance $\lesssim w$ from $[-w,w]$.  Thus off the union of
$\leq D$ intervals $I_i$ of length $\ll D^{-1}w$, $|P| \gtrsim C_D
\|P\|_{L^{\infty}([-w,w])}$.  The intersection of $S$ with this
union is small by centrality, so we may take $S' = S\backslash
\bigcup_i I_i$. 
\end{proof}

In particular, if we take $S = [-w,w]$, we see that
\[
    \|P\|_{L^{\infty}([-w,w])} \sim w^{-1} \int_{[-w,w]}|P|.
\]

\begin{lemma} \label{polyTn}
    If $T_i \subset \reals^i$, $1 \leq i \leq n$, if $T_{i+1}
    \subset T_i \times [-1,1]$, $1 \leq i \leq n-1$, and if the sets
    \[
        \tau_1:=T_1 \qquad
        \tau_i(t):=\{s \in [-C,C]:(t,s) \in T_i\}
    \]
    are central sets of width $w_1$ and $w_i$ for each $t \in
    T_{i-1}$, respectively, and if $P$ is a polynomial of degree $D$
    on $\reals^n$, then
    \[
        |P| \gtrsim C_D \|P\|_{L^{\infty}(\prod_{i=1}^n [-w_i,w_i])}
    \]
    on a subset $T_n' \subset T_n$ of measure $\gtrsim |T_n|$.
\end{lemma}

\begin{proof}  We briefly recount the proof of this from \cite{ChRL}.  Its proof from
the previous lemma is as follows.  Considering $P^2$ if necessary,
we may assume that $P \geq 0$.

We define polynomials $P_i$ on $\reals^i$, $1 \leq i \leq n$ of
degree $\leq D$ as follows:

Let $P_n:=P$, and for $1 \leq i \leq n-1$ and $t \in \reals^i$,
define
\[
    P_i(t) = w_{i+1}^{-1}\int_{[-w_{i+1},w_{i+1}]}P_{i+1}(t,s)ds.
\]
For each $i$, $2 \leq i \leq n$ and for each $t \in T_{i-1}$,
\[
    \tau_i'(t):= \{s \in \tau_i(t):P_i(t,s) \gtrsim C_D
    \|P_i(t,\cdot)\|_{L^{\infty}([-w_i,w_i])}\}
\]
has measure $\gtrsim |\tau_i(t)|$ (because $s \mapsto P_i(t,s)$ is a
polynomial of degree $\leq D$).  Let
\[
    T_1':=\{s \in T_1:|P_1(s)| \gtrsim
    C_D\|P_1\|_{L^{\infty}([-w_1,w_1])}\}
\]
and define the sets
\[
    T_i':= \{(t,s):t \in T_{i-1}' \text{ and } s \in \tau_i'(t)\}.
\]
Then for each $i$, $|T_i'| \sim |T_i|$, and if $t \in T_n'$,
\[
    P(t) \sim C_D\|P\|_{L^{\infty}(\prod_{i=1}^n [-w_i,w_i])}.
\]
\end{proof}

\subsection{Back to $\Phi_{\bfj}^n$.}  
In this subsection, we use the polynomial lemmas to make the heuristic arguments from an earlier subsection rigorous and thereby complete the proof of Theorem~\ref{suff}.  The main step will be to prove that the mapping $\Phi_{\bfj}^n$ is $O(N)$-to-1 off of a negligible set.  This, when combined with the lemmas on generic sets, will give us the lower bound on the volume of $\Phi_{\bfj}^n(T_n)$ that we need.

Because $\Phi_{\bfj}^n$ is smooth, from the definition of $B_{\bfj}(x_0;\delta_1,\ldots,\delta_k)$, we have the upper bound
\begin{align} \label{ub_Bj}
|B_{\bfj}(x_0;\delta_1,\ldots,\delta_k)| &= |\Phi_{\bfj}^n(\tilde{T}_n)| \leq \int_{\tilde{T}_n} |\det(\partial_t \Phi_{\bfj}^n(t))|\, dt \\ \notag & \leq \delta^{\deg(\bfj)} \| \det \partial_t \Phi_{\bfj}^n\|_{L^{\infty}(\tilde{T}_n)} =: \delta^{\deg(\bfj)} \scriptJ_{\bfj}(\delta).
\end{align}
In addition, we have already seen that 
\begin{align} \label{compare_balls}
|B_{\bfj}(x_0;\delta_1,\ldots,\delta_k)| \sim |B(x_0;\delta_1,\ldots,\delta_k)| \sim |\Lambda_{\delta}(x_0)|.
\end{align}
Combining \eqref{ub_Bj} and \eqref{compare_balls}, we thus have the inequality
\[
\tfrac{1}{\delta^{\deg(\bfj)}}|\Lambda_{\delta}(x_0)| \lesssim \scriptJ_{\bfj}(\delta).
\]
Because $\Phi^n_{\bfj}$ is smooth, $\scriptJ_{\bfj}(\delta)$ is bounded above, and because $|\Lambda_{\delta}(x_0)| \gtrsim \delta^{\deg(I_0)}$, $\scriptJ_{\bfj}(\delta)$ is bounded below.  In particular, we have
\begin{align} \label{lb_Jj}
\alpha_k^C \lesssim \scriptJ_{\bfj}(\delta) \lesssim 1
\end{align}
(by the definition of the $\delta_j$ and monotonicity of the $\alpha_j$).

We now use this information to establish a lower bound for $|\Phi_{\bfj}^n(T_n)|$.  This will complete the proof of the theorem.  Let $Q_w := \prod_{j=1}^n [-w_j,w_j]$ be the rectangle whose side-lengths are given by the widths $w_j$ (so $T_n \subset Q_w$).  Then smoothness of $\Phi_{\bfj}^n$, the lower bound \eqref{lb_Jj}, and the fact that $w_j \lesssim \alpha_k^{\eps}$ imply that we can control high order remainder terms in the Taylor series of $\Phi_{\bfj}^n$ on $Q_w$.  In fact, taking $N \gg \eps^{-1}$ and the $\alpha_j$ sufficiently small, we have
\begin{align} \label{bound_diff}
\|\Psi_{\bfj} - \Phi_{\bfj}^n\|_{C^2(Q_w)} \leq c_0 \alpha_k^C(\scriptJ_{\bfj}(\delta))^2,
\end{align}
where $\Psi_{\bfj}$ is the Taylor polynomial of $\Phi_{\bfj}^n$ of degree $N$ centered at 0, and we may choose $c_0$ as small and $C$ as large as we like.  ($N$ will also depend on $C,c_0$.)  Here we may choose the $\alpha_j$ as small as needed (how small depends on $b,\eps, N$, and $\|X_j\|_{C^{M_N}(V)}$, $j=1,\ldots,k$) by making the initial partition of unity sufficiently fine. 

We apply a linear transformation $A_w$, mapping $Q_w$ onto the unit cube $Q$.  The images of the $\tau_j$, denoted $\tilde{\tau}_j$, are then central sets of width 1, and, denoting by $\tilde{\Psi}$ and $\tilde{\Phi}$ the compositions $\Psi_{\bfj} \circ A_w^{-1}$ and $\Phi_{\bfj}^n \circ A_w^{-1}$ (respectively), we have that $\tilde{\Psi}$ and $\tilde{\Phi}$ are smooth, that $\tilde{\Psi}$ is the degree $N$ Taylor polynomial of $\tilde{\Phi}$, and that the bound \eqref{bound_diff} is transformed to 
\begin{align} \label{bound_diff2}
\|\tilde{\Psi} - \tilde{\Phi}\|_{C^2(Q)} \leq c_0 (\tilde{\scriptJ})^2,
\end{align}
where $\tilde{\scriptJ} = \| \det \partial_t \tilde{\Phi}\|_{L^{\infty}(Q)}$.

We are now in a position to apply Lemma 7.1 of \cite{ChRL}, the upshot of which we state below.  

\begin{lemma} \label{lemma_7_1}
The bound \eqref{bound_diff2} and the centrality of the intervals $\tilde{\tau}_j$ imply that
\begin{align} \label{Lemma7_1}
|\tilde{\Phi}(A_w(T_n))| \gtrsim |A_w T_n| \tilde{\scriptJ} \gtrsim \tfrac{|A_w T_n|}{|A_w \tilde{T}_n|} \int_{A_w \tilde{T}_n} |\det \partial_t \tilde{\Phi}(t)|\, dt.
\end{align}
\end{lemma}

Undoing the linear transformation, this implies that
\[
|\Phi_{\bfj}^n(T_n)| \gtrsim \tfrac{|T_n|}{|\tilde{T}_n|} \int_{\tilde{T}_n} |\det \partial_t \Phi_{\bfj}^n(t)|\, dt \gtrsim \alpha_k^{C\eps}|B_{\bfj}(x;\delta_1,\ldots,\delta_n)|,
\]
and the theorem is proved.

We briefly sketch the proof of Lemma~\ref{lemma_7_1}.  

\begin{proof}[Sketch of proof]  First, we may use the bound \eqref{bound_diff2} and centrality of the $A_w(\tau_j)$ to refine $A_w(T_n)$ to the region where 
\[
|\det \partial_t \tilde{\Psi}(t)| \sim |\det \partial_t \tilde{\Phi}(t)| \gtrsim \tilde{\scriptJ},
\]
without significantly reducing the volume.  Denote the refined region by $\omega$.  The idea is to show that $\tilde{\Phi}$ is $O(N)$-to-1 on $\omega$, which implies that
\[
|\tilde{\Phi}(A_w(T_n))| \geq |\tilde{\Phi}(\omega)| \geq \tfrac{1}{C_N} \int_{\omega} |\det \partial_t \tilde{\Phi}(t)|\, dt;
\]
\eqref{Lemma7_1} follows.

By quantitative forms of the inverse function theorem, $\tilde{\Phi}$ is injective on balls of diameter less than a small constant times $\tilde{\scriptJ}$.  We cover $\omega$ by a family $\{Q_j\}$ of such balls.  The next step is to use the polynomial approximation to complete the argument.

Let $Q_j^*$ denote the dilate of $Q_j$ by a factor $C_d$.  Then these cubes have bounded overlap, and by Bezout's theorem from algebraic geometry, a point $y \in \reals^n$ may lie in the image of at most $C_{N,d}$ of these cubes under the (polynomial) mapping $\tilde{\Psi}$.  The final step of the proof is to transfer this property to $\tilde{\Phi}$ by showing that $\tilde{\Phi}(Q_n) \subset \tilde{\Psi}(Q_n^*)$.  

The argument is topological.  Let $y \in \tilde{\Phi}(Q_j)$ and let $B$ be a ball with $Q_j \subset B \subset Q_j^*$.  Local injectivity and the bound \eqref{bound_diff2} imply that $y \notin \tilde{\Psi}(\partial B) \cup \tilde{\Phi}(\partial B)$ and that
\[
\tilde{\Psi}|_{\partial B} : \partial B \to \reals^n \backslash \{y\} \qquad \tilde{\Phi}|_{\partial B} : \partial B \to \reals^n \backslash \{y\}
\]
are homotopic.  By local injectivity, their topological degree must then be 1, which implies in particular that $y \in \tilde{\Psi}(B)$.  Further details may be found in \cite{ChRL}.
\end{proof}

\end{document}